\documentclass[12pt]{amsart}

\usepackage{amssymb}
\usepackage{amsfonts,verbatim}
\newtheorem{theorem}{Theorem}[section]
\newtheorem{proposition}[theorem]{Proposition}
\newtheorem{corollary}[theorem]{Corollary}
\newtheorem{lemma}[theorem]{Lemma}

\theoremstyle{remark} \newtheorem{remark}[theorem]{Remark}

\newcommand\C{\mathbb{C}}

\newcommand\R{\mathbb{R}}

\newcommand\RR{\widetilde{\R}}

\newcommand\Dir{\operatorname{Dir}}
\newcommand\Neu{\operatorname{Neu}}
\newcommand\LapoD{\Delta^{(0, \infty)}_{\Dir, d}}
\newcommand\LapoN{\Delta^{(0, \infty)}_{\Neu, d}}
\newcommand\LapiD{\Delta^{(1, \infty)}_{\Dir, d}}
\newcommand\LapiN{\Delta^{(1, \infty)}_{\Neu, d}}
\newcommand\TiD{T^{(1, \infty)}_{\Dir, d}}
\newcommand\TiN{T^{(1, \infty)}_{\Neu, d}}

\newcommand\Td{\widetilde{T_d}}

\newcommand\LaprD{\Delta^{\R}_{ d}}
\newcommand\Lapd{\widetilde{\Delta_d}}
\newcommand\Lapda{\widetilde{\Delta_{d,a}}}

\newcommand\Lpz{L^p((0,\infty);r^{d-1}dr)}
\newcommand\Lpj{L^p((1,\infty);r^{d-1}dr)}

\newcommand\Lpwr{L^p( \RR :r^{d-1}dr)}
\newcommand\Dom{\operatorname{Dom}}

\newcommand\Id{\operatorname{Id}}

\begin{document}

\title[Riesz transforms in one dimension ]{Riesz transforms in one dimension }
\author{Andrew Hassell and Adam Sikora}
\address{Andrew Hassell, Department of Mathematics, Australian National University,
ACT 0200 Australia
}
\email{hassell@maths.anu.edu.au}
\address{Adam Sikora, Department of Mathematical Sciences, New Mexico State
University, Las Cruces, NM 88003-8001, USA
and Department of Mathematics,
Australian National University,
ACT 0200 Australia
}
\email{asikora@nmsu.edu\textrm{, } sikora@maths.anu.edu.au}


\begin{abstract}
We study the boundedness on $L^p$ of the Riesz transform $\nabla L^{-1/2}$, where $L$ is one of several operators defined on $\R$ or $\R_+$, endowed with the measure $r^{d-1} dr$, $d > 1$, where $dr$ is Lebesgue measure. For integer $d$, this mimics the measure on Euclidean $d$-dimensional space, and in this case our setup is equivalent to looking at the Laplacian acting on radial functions on Euclidean space or variations of Euclidean space such as the exterior of a sphere (with either Dirichlet or Neumann boundary conditions), or the connected sum of two copies of $\R^d$. In this way we illuminate some recent results on the Riesz transform on asymptotically Euclidean manifolds. 

We are however interested in all real values of $d > 1$, and another goal of our analysis is to study the range of boundedness as a function of $d$; it is particularly interesting to see the behaviour as $d$ crosses $2$. For example, in one of our cases which models radial functions on the connected sum of two copies of $\R^d$, the upper threshold for $L^p$ boundedness is $p=d$ for $d \ge 2$ and $p=d/(d-1)$ for $d < 2$. Only in the case $d=2$ is the Riesz transform actually bounded on $L^p$ when $p$ is equal to the upper threshold. 

We also study the Riesz transform when we have an inverse square potential, or a delta function potential; these cases provide a simple model for recent results of the first author and Guillarmou. Finally we look at the Hodge projector in a slightly more general setup.
% (still in one dimension). 
\end{abstract}

\maketitle

\section{introduction}

Using elementary calculations based on modified Bessel functions, we obtain a complete description of $L^p$-continuity properties of the Riesz transform $\nabla L^{-1/2}$ for several families of Laplace type operators $L$ defined on $\R$ or $\R_+$, with respect to the measure $|r|^{d-1} dr$ where $dr$ is Lebesgue measure. This measure mimics the measure on Euclidean space $\R^d$  for $d = 1, 2,  3, \dots$ and for that reason we refer to $d$ as the `dimension'. However we consider all real $d > 1$ in this paper. We aim rather for completeness and simplicity then generality
of results. However our results are a  good model for considering  a wide range of
multidimensional Riesz transforms. In fact many surprising  negative results for the Riesz transform follows from results in this note. For example, 

\begin{itemize}

\item
Proposition~\ref{prop-c} shows that the Riesz transform for the operator $\Delta + c/|x|^2$ on $\R^d$,  $d \geq 3$, with $-(d/2 - 1)^2 < c < 0$ cannot be bounded on $L^p(\R^d)$ unless $p$ is in the interval 
$$\Big( \frac{d}{ \frac{d}{2} + 1 + \sqrt{(\frac{d}{2} - 1)^2 + c}}, \ \frac{d}{\frac{d}{2} - \sqrt{(\frac{d}{2} - 1)^2 + c})} \Big),$$ 
and our results suggest that boundedness holds precisely in this range (see Section~\ref{isqp} and especially Remark~\ref{ccc});

\item
Theorem~\ref{main} shows that the Riesz transform for the Dirichlet Laplacian on $\R^d \setminus B(0,r)$, $d \geq 2$ is unbounded on $L^p(\R^d)$ for $p \geq d$ ($p > d$ for $d=2$), and our results suggest that it is most likely bounded for all $1 < p < d$ (see Remark~\ref{B(0,r)}). 
\end{itemize}

We expect that our results  govern the multidimensional theory; by this we mean that the range  of $L^p$ spaces on which the Riesz transforms are bounded coincides with the range calculated here in the one dimensional case.  
We also expect that  our results could be used as an important step in proof of such multidimensional generalizations.

The main result obtained in this note can be described in the following way. 
For $d>1$
consider the space $L^2(\R, (1+|r|)^{d-1}dr)$ and the operator
$L = \nabla^* \nabla$, where $\nabla f = f'$ is the derivative operator and $\nabla^*$ is the adjoint with respect to the measure $(1 + |r|)^{d-1}$. Then we have
\begin{theorem}\label{mr} Let $L$ be as above. 
The Riesz transform $d L^{-1/2}$  is bounded on $L^p(\R, (1+|r|)^{d-1}dr)$ if and only if
\begin{enumerate}
\item[(i)] $1<p<d$ for $d>2$
\item[(ii)] $1<p \le 2$ for $d=2$
\item[(iii)] $1<p < \frac{d}{d-1}$ for $1 < d < 2$.
\end{enumerate}
\end{theorem} 
For integer $d$, the operator $L$ models the radial part of Beltrami-Lpalce operator acting on two copies of $\R^d \setminus B(0,1)$ glued together on the boundary. 
In equivalent notation if $\delta+1/d=1$ and $\nabla^\delta = (1 + |r|^\delta) d$,  then the Riesz transform $\nabla^\delta ((\nabla^\delta)^* \nabla^\delta)^{-1/2}$ 
 is bounded on $L^p(\R, \,  dr)$ for the same range of $p$ as in Theorem~\ref{mr}.

The boundedness of the Riesz transform is one of central points of harmonic analysis and the theory of partial differential equations. The investigation of the classical Riesz transform initiated the development of the theory of singular integrals, see \cite{Ri, CZ}. 
In 1983 \cite{Str} Strichartz asked about sufficient condition for continuity on $L^p$ spaces of the Riesz transform on complete Riemannian manifolds.  
In other terminology this is a question about equivalence of two possible
definitions of Sobolev $L^p$ spaces on Riemannian manifolds. 
This clearly significant problem turns out to be surprisingly complex. Despite
being investigated by several authors, see for example \cite{ACDH, CCH, Li} and references within,
there are few specific setting for which boundedness of the Riesz transform on $L^p$ spaces is completely described. In particular, while there is a reasonably general positive result for $p < 2$ \cite{CD}, rather little appears to be known about boundedness on $L^p$ for $p > 2$. In this context the results described in this paper extend the family of fully understood  examples of the Riesz transform in a significant way.

There are two main approaches in studies of Riesz transform: probabilistic and analytic. The analytic methods are related to the theory of singular integrals. 
For more background information on the analytic method for the Riesz transform we refer reader to \cite{ACDH, CD} and to references within. 
For a description of probabilistic approach we refer the readers to \cite{St, Ba2}. 
We use only (very elementary) analytic methods but we would like to mention  two other papers, which use probabilistic approach \cite{Ba1, Ro}. Both these works are devoted to the one-dimensional case; that is the case that the underlying space is equal to the real line, as is the case in this note. Surprisingly  there is no connection between the results, which we obtain and those discussed in \cite{Ba1, Ro}.

 Similarly as in \cite{CCH, GH1, GH2, Li, Sh1, Sh2} this note studies the range of $p$ for which Riesz transform is bounded on  Lebesgue $L^p$ space.  
Our results  are motivated by  \cite{CCH} and \cite{GH1, GH2}. In \cite{CCH} the boundedness of the Riesz transform is studied in the setting of a Riemannian manifold which is the union of a compact part and two Euclidean Ends,  $\R^d \setminus B(0,r)$ for some $r>0$. Roughly speaking the operators, which we study here are the Riesz transform on such manifolds restricted to radial functions. It seems that the radial part is most essential for understanding the general behaviour of the Riesz transform. Considering the radial part only allow us to investigate `fractional' dimensions $d$ and  allows us to observe new possible phenomena in the behaviour of the Riesz transform, which are especially interesting for dimension $1<d<2$.  
In fact collecting a large class of different behaviours for the Riesz transform is one of the main goals of this note.

%%%%%%%%%%%%%%%%%%%%%%%%%%%%%%%%%%%%

\section{Various one-dimensional operators}\label{opdef}

In this section, we define several operators on $L^2(*, |r|^{d-1} dr)$, where $*$ is either the real line $\R$, the half line $[0,  \infty)$, the ray $[1, \infty)$ or the `broken' line $\tilde \R = (-\infty, -1] \cup [1, \infty)$. 

\subsection{Operators $\LapoD $ acting on $L^2((0,\infty), r^{d-1}dr)$}

For $d > 1$  we consider the space $L^2(\R_+,r^{d-1}dr)$. For  $f,g \in C_c^\infty(0,\infty)$ we define the quadratic form 
\begin{equation}
Q_{d}^{(0,\infty)}  (f,g)=\int^{\infty}_{0}f'(r)g'(r)r^{d-1} dr.
\label{Q1}\end{equation}
Using the Friedrichs extension one can define  $ \LapoD $ as the unique self-adjoint operator corresponding to $ Q_{d}^{(0,\infty)} $, acting on $L^2(\R_+,r^{d-1}dr)$ and formally given by the following formula
$$
\LapoD  f=-\frac{d^2}{dr^2}f-\frac{d-1}{r}\frac{d}{dr}f.
$$
Note that the canonical gradient (defined using the notion of carr\'e du champ, see  \cite{BH}) corresponding to $\LapoD $  is given by 
$$
|\nabla f|^2= \frac{1}{2}(\LapoD  f^2 -2f  \LapoD f)=|f'|^2.
$$

\subsection{Operators $ \LapoN $ acting on $L^2(r^{d-1}dr,(0,\infty))$} 
For $d>1$ we define the Neumann Laplacian $   \LapoN  $ acting on $L^2(\R_+,r^{d-1}dr)$
using the same quadratic form  $Q_{d}^{(0,\infty)} $ as in \eqref{Q1}, but with a different domain. Let $\phi \in C_c^\infty[0, \infty)$ be identically equal to $1$ in a neighbourhood of $0$. To  impose the Neumann boundary condition we  initially define the quadratic form on $C^\infty_c(0,\infty) \oplus \mathbb{C} \phi$ (instead of $C^\infty_c(0,\infty)$). (We remark that we can use the domain $C_c^\infty[0, \infty)$ for $d > 2$, but not for $d \leq 2$ since this space is not contained in the domain of the operator, as $1/r \notin L^2$ locally near $r=0$ for $d \leq 2$.) 
For $d\ge 2$, the operators $\LapoN$ and $\LapoD$ coincide;  see \cite{ERSZ}. Below we consider the $\LapoN$ for all $d > 1$ and $\LapoD$ only for $1 < d < 2$.

\subsection{Operators $\LapiD $ and  $\LapiN  $  acting on 
$L^2((1,\infty), r^{d-1}dr)$} 
We define $\LapiD $ and  $\LapiN $ as  the Dirichlet, resp. Neumann  extensions of the quadratic form 
$$
Q_{d}^{(0,\infty)} (f,g)=\int^{\infty}_{1}f'(r)g'(r)r^{d-1} dr.
$$
That is for Dirichlet operator $\LapiD$, resp. Neumann operator   $\LapiN $,  we take the closure of the above form initially defined on $C^\infty_c(1,\infty)$, resp.  $C^\infty_c[1,\infty)$.

\subsection{Operators $\Lapd $ acting on $L^2(r^{d-1}dr,\widetilde{\R})$ where $\widetilde{\R}= (-\infty,-1] \cup [1,\infty)$.}\label{seclap}

We consider the set $\widetilde{\R}= (-\infty,-1] \cup [1,\infty)$. 
%If $X\subset \widetilde{\R}$ then by $f|_X$ we denote the restriction of $f$ to the set $X$.  
We say that  $f\in C^1(\widetilde{\R})$ if $f$ is $C^1$ on the intervals $(-\infty, -1]$ and $[1, \infty)$ and if 
%$f|_{(-\infty,-1]}\in C^1((-\infty,-1])$,$f|_{[1,\infty)}  \in C^1([1,\infty))$,
 $f(-1)=f_+(1)$ and  if $f'(-1)=f_+'(1)$. 
For $f,g \in C^1(\widetilde{\R})$ we define the quadratic form 
\begin{equation}
\widetilde{Q_d}(f,g)=\int_{-\infty}^{-1}f'(r)g'(r)|r|^{d-1} dr
+\int_1^{\infty}f'(r)g'(r)r^{d-1} dr.
\label{Q2}\end{equation}
 Note that if $\Lapd$ is the unique self-adjoint operator corresponding to $\widetilde{Q_d}$
then 
\begin{equation}
\Lapd f=-\frac{d^2}{dr^2}f-\frac{d-1}{r}\frac{d}{dr}f.
\label{form}\end{equation}
Note also that the operators $\Lapd$ are equivalent to the operators corresponding to the following quadratic form 
$$
Q'(f,g)=\int_{-\infty}^{\infty}f'(r)g'(r)(1+|r|)^{d-1} dr
$$
acting on $L^2((1+|r|)^{d-1}dr,{\R})$, as in Theorem~\ref{mr}. However, it significantly simplifies notation to define the operator on $\tilde \R$ as above. 

\subsection{ Operators   $ \LaprD  $ acting on $L^2( |r|^{d-1}dr,\R)$}
For $1<d<2$, we can define the operator above but on the domain $(-\infty, -\epsilon] \cup [\epsilon, \infty)$, or equivalently the quadratic form $Q'$ using the measure $(\epsilon+|r|)^{d-1} dr$ instead of $(1+|r|)^{d-1} dr$. It follows form a result of Kato, \cite[Theorem VIII.3.11]{Ka}, that this sequence of operators has a limit as $\epsilon \to 0$ in the strong resolvent sense, as described in \cite{ERSZ}. We denote this limit operator by $\LaprD$; it is given formally by the 
formula \eqref{form}, with $r \in \R$. %, on functions defined on $\R$, with a certain domain that we do not specify here. 
 For $d\ge 2$ the operator $   \LaprD  $ 
is equal to direct sum of two copies of $  \LapoD=  \LapoN$ --- see \cite{ERSZ}. 
Hence there is no point to considering $   \LaprD $ separately for $d \ge 2$.

%%%%%%%%%%%%%%%%%%%%%%%%%%%%%%%%%%%%

\section{Special functions}\label{specfns}

\subsection{The functions $k$ and $l$}

We will compute an exact formula for the kernel of the resolvent $(L+\lambda^2)^{-1}$, for all of the operators $L$ defined in Section~\ref{opdef}, in terms of special functions $k$ and $l$, closely related to modified Bessel functions.  
Consider the following ordinary differential equation 
\begin{equation}
f''+\frac{d-1}{r}f'= f.
\label{kl}\end{equation}
We set 
$$
{F}(r)=r^{d/2-1}f(r) \quad \mbox{i.e.} \quad f(r)={F}(r)r^{1-d/2}.
$$
%(why -1?)
Then 
\begin{eqnarray*}
({F}r^{1-d/2})''+\frac{d-1}{r} ({F}r^{1-d/2})'
-{F}r^{1-d/2}=0.
\end{eqnarray*}
This simplifies to 
\begin{eqnarray*}
r^2{F}''+{r} {F}'-(r^2+(d/2-1)^2){F}=0.
\end{eqnarray*}
Hence ${F}$ is a combination of modified Bessel functions $I_{d/2- 1
}(r)$ and $K_{|d/2-1|}(r)$, 
see \cite[\S 9.6.1 p. 374]{AS} or \cite[\S 1.14 p. 16]{Tr}.
Now we note that 
any solution of the equation
\begin{equation}\label{wuj}
f''+\frac{d-1}{r}f'= \lambda^2f.
\end{equation}
is a linear combination of
the functions $r \to l_d(\lambda r)$ and $ r \to k_d(\lambda r)$, where 
$$
l_d(r)=r^{1-d/2}I_{d/2- 1}(r)
\quad  \mbox{and} \quad k_d(r)=r^{1-d/2}K_{|d/2- 1|}(r).
$$ 
In the sequel we going to skip index $d$ in our notation; that is, we use just $l$ and $k$ instead of $l_d$ and $k_d$.

Next we compute the Wronskian $W(r)=l(r)k'(r)-l'(r)k(r)$ corresponding to the equation 
(\ref{wuj}).
Note that by (\ref{wuj}) % (with $\lambda=1$)
 \begin{eqnarray*}
(lk'-kl')'=  lk''-kl''= -
\frac{\mu'}{\mu}(lk'-kl'),
\end{eqnarray*}
where $\mu(r)=r^{d-1}$. Hence
\begin{eqnarray*}
[\ln (lk'-kl')]'=-[\ln(\mu)]'
\end{eqnarray*}
and so
\begin{equation}\label{wron}
lk'(r)-kl'(r)=\frac{1}{\nu r^{d-1}},
\end{equation}
where the constant $\nu$ depends on $d$ but does not depend on $r$. 

Finally we define $A = l/k$, $B = l'/k'$, $C = A + B$ and $D = A-B$. Our kernels in Section~\ref{res-kernels} will be written in terms of the functions $k, l, k', l'$, $A, B, C$ and $D$. 

\subsection{Positivity properties}

In order to get bounds on the kernels of our resolvents $(L+\lambda^2)^{-1}$, we
need information on the positivity of $k$, $l$ and associated functions. 

\begin{lemma}\label{pos} For all $d > 1$, each of the functions  $k$, $l$, $k'$, $l'$, $A$, $B$, $D$ and $r k'(r) +(d-2) k(r)$ has a fixed sign on $(0, \infty)$. \end{lemma}

\begin{proof} Suppose, for a contradiction, that $k'$ has a zero at $\lambda_0 > 0$. Without loss of generality we may suppose that $k(\lambda_0) > 0$. Then from the equation \eqref{kl} we see that $k''(\lambda_0) > 0$. But $k(\lambda) \to 0$ as $\lambda \to \infty$, so there must be a maximum value of $k$ at $\lambda_1 > \lambda_0$, i.e. a $\lambda_1 > \lambda_0$ where $k(\lambda) > 0, k'(\lambda) = 0$ and $k''(\lambda) \leq 0$. This however is impossible in view of \eqref{kl}. It follows immediately that $k$ has no positive zero. 

As for $l(\lambda)$, from the equation and the fact that  $l(0) > 0$ we see that $l'(\lambda) > 0$ for small $\lambda$. If there were a positive zero of $l'$, then let $\lambda_0$ be the first such zero. We would then have $l(\lambda_0) > 0$, $l'(\lambda_0) = 0$, $l''(\lambda_0) \leq 0$, which is impossible from \eqref{kl}. It follows immediately that $l$ also has no zero. 

Since $A = l/k$, $B = l'/k'$ we see that neither $A$ nor $B$ change sign. Moreover, 
$D = (lk'-kl')/(kk')$ which does not change sign since the numerator is the Wronskian $cr^{-d+1}$. Finally, from the equation \eqref{kl} we deduce
$$
\big(rk'(r) + (d-2)k(r)\big)' = rk(r)
$$
which has a fixed sign. The asymptotics (A), (B), (C) imply that $rk'(r) + (d-2) k(r) \to 0$ as $r \to \infty$, hence $rk'(r) + (d-2) k(r)$ has a fixed sign. 
\end{proof}

\subsection{Asymptotic behaviour of the functions $l$ and $k$}\label{asym}
In the sequel we use some standard asymptotics for the functions $f$ and $k$ which we describe below. 
For proofs of these results we refer readers to   \cite[\S 9.6.1 p. 374]{AS} or \cite[\S 1.14 p. 16, \S 3.6, 3.7 p. 49, 50]{Tr}. 

Then we have (where $f \approx w$ means that there exist positive constants $c, C$ such that $cw \leq f \leq Cw$)
\begin{enumerate}
\item[(A)]\label{AA} For $d>2$
\begin{equation*} 
\begin{gathered}
k(\lambda)\approx \left\{ \begin{array}{ll}
\lambda^{2-d}  & \mbox{if} \quad \lambda \le 1\\
\lambda^{(1-d)/2} e^{-\lambda}    &    \mbox{if} \quad \lambda > 1
 \end{array}
    \right.
\\
k'(\lambda)\approx \left\{ \begin{array}{ll}
-\lambda^{1-d}  & \mbox{if} \quad \lambda \le 1\\
-\lambda^{(1-d)/2}e^{-\lambda}    &    \mbox{if} \quad  \lambda >1
 \end{array}
    \right.
\end{gathered} \qquad
\begin{gathered}
l(\lambda)\approx \left\{ \begin{array}{ll}
1  & \mbox{if} \quad \lambda \le 1\\
\lambda^{(1-d)/2}e^{\lambda}    &    \mbox{if} \quad 1 \le \lambda
 \end{array}
    \right.
\\
l'(\lambda)\approx \left\{ \begin{array}{ll}
\lambda  & \mbox{if} \quad \lambda \le 1\\
\lambda^{(1-d)/2}e^{\lambda}    &    \mbox{if} \quad1\le \lambda
 \end{array}
    \right.
\end{gathered}
\end{equation*} 
\begin{equation*} 
A(\lambda) \approx D(\lambda)\approx \left\{ \begin{array}{ll}
\lambda^{d-2}  & \mbox{if} \quad \lambda \le 1\\
e^{2\lambda}    &    \mbox{if} \quad1\le \lambda
 \end{array}
    \right.
\end{equation*} 
\begin{equation*} 
B(\lambda)\approx \left\{ \begin{array}{ll}
-\lambda^{d}  & \mbox{if} \quad \lambda \le 1\\
-e^{2\lambda}    &    \mbox{if} \quad1\le \lambda;
 \end{array}
    \right.
\end{equation*} 
%%%%%%%%%%%%%%%%
\item[(B)]\label{BB} For $d=2$
\begin{equation*} 
\begin{gathered}
k(\lambda)\approx \left\{ \begin{array}{ll}
-{\log(\lambda)}  & \mbox{if} \quad \lambda \le 1\\
\lambda^{-1/2}e^{-\lambda}    &    \mbox{if} \quad \lambda > 1
 \end{array}
    \right.
\\
k'(\lambda)\approx \left\{ \begin{array}{ll}
-{\lambda}^{-1} & \mbox{if} \quad \lambda \le 1\\
-\lambda^{-1/2}e^{-\lambda}    &    \mbox{if} \quad  \lambda >1
 \end{array}
    \right.
\end{gathered}  \qquad
\begin{gathered} 
l(\lambda)\approx \left\{ \begin{array}{ll}
1  & \mbox{if} \quad \lambda \le 1\\
\lambda^{-1/2}e^{\lambda}    &    \mbox{if} \quad1\le \lambda
 \end{array}
    \right.
\\
l'(\lambda)\approx \left\{ \begin{array}{ll}
\lambda  & \mbox{if} \quad \lambda \le 1\\
\lambda^{-1/2}e^{\lambda}    &    \mbox{if} \quad1\le \lambda
 \end{array}
    \right.
\end{gathered}\end{equation*} 
\begin{equation*} 
A(\lambda)\approx D(\lambda) \approx \left\{ \begin{array}{ll}
\frac{-1}{\log(\lambda)}  & \mbox{if} \quad \lambda \le 1\\
e^{2\lambda}    &    \mbox{if} \quad1\le \lambda
 \end{array}
    \right.
\end{equation*} 
\begin{equation*} 
 B(\lambda) \approx \left\{ \begin{array}{ll}
-\lambda^{2} & \mbox{if} \quad \lambda \le 1\\
-e^{2\lambda}    &    \mbox{if} \quad1\le \lambda;
 \end{array}
    \right.
\end{equation*} 
\bigskip
%%%%%%%%%%%%%%%%%
\item[(C)]\label{CC} For $d<2$
\begin{equation*} 
\begin{gathered}
k(\lambda)\approx \left\{ \begin{array}{ll}
1  & \mbox{if} \quad \lambda \le 1\\
\lambda^{(1-d)/2}e^{-\lambda}    &    \mbox{if} \quad \lambda > 1
 \end{array}
    \right.
\\
k'(\lambda)\approx \left\{ \begin{array}{ll}
-\lambda^{1-d}  & \mbox{if} \quad \lambda \le 1\\
-\lambda^{(1-d)/2}e^{-\lambda}    &    \mbox{if} \quad  \lambda >1
 \end{array}
    \right.
\end{gathered}  \qquad
\begin{gathered} 
l(\lambda)\approx \left\{ \begin{array}{ll}
1  & \mbox{if} \quad \lambda \le 1\\
\lambda^{(1-d)/2}e^{\lambda}    &    \mbox{if} \quad1\le \lambda
 \end{array}
    \right.
\\
l'(\lambda)\approx \left\{ \begin{array}{ll}
\lambda & \mbox{if} \quad \lambda \le 1\\
\lambda^{(1-d)/2}e^{\lambda}    &    \mbox{if} \quad1\le \lambda
 \end{array}
    \right.
\end{gathered}\end{equation*} 
\begin{equation*} 
A(\lambda)\approx D(\lambda)\approx\left\{ \begin{array}{ll}
1  & \mbox{if} \quad \lambda \le 1\\
e^{2\lambda}    &    \mbox{if} \quad1\le \lambda.
 \end{array}
    \right.
\end{equation*} 
\end{enumerate}

\begin{equation*} 
B(\lambda)\approx \left\{ \begin{array}{ll}
-\lambda^{d}  & \mbox{if} \quad \lambda \le 1\\
-e^{2\lambda}    &    \mbox{if} \quad1\le \lambda.
 \end{array}
    \right.
\end{equation*} 

\begin{remark}\label{asympt} Much more is true: the functions $k$, $l$ have complete conormal expansions as $\lambda \to 0$, and the functions $k(\lambda) \lambda^{(d-1)/2} e^\lambda$ and $l(\lambda) \lambda^{(d-1)/2} e^{-\lambda}$ have complete expansions in negative powers as $\lambda \to \infty$. We do not need these expansions except in the proof of Theorem~\ref{rton}, where we need the asymptotic $f(\lambda) \lambda^{(d-1)/2} e^\lambda \sim c + O(\lambda^{-1})$ for $f = k, l, k', l'$ and some constant $c$. 

We also note that an upper bound for $C$ is obtained by adding the bounds for $A$ and $B$; however, $C$  changes sign so there is no corresponding lower bound. 
\end{remark}

%%%%%%%%%%%%%%%%%%%%%%%%%%%%%%%%%%%

\section{Resolvent kernels}\label{res-kernels}
In this section we compute the exact kernel of the resolvent for each of the operators of the previous section. 

\subsection{The resolvent for $\LapoD$ and $\LapoN$} 
We first analyze the domains of these operators. Recall that they coincide unless $d < 2$. To determine the different resolvent kernels for $d < 2$, we use the following lemma. 

\begin{lemma}\label{dom}  For $1 < d < 2$ the domain of $\LapoD$ satisfies
$$
u \in \Dom \LapoD \implies u(r) = O(r^{(2-d)/2}) \text{ as } r \to 0,
$$
while the domain of $\LapoN$ satisfies 
$$
u \in \Dom \LapoN \implies u'(r) = o(r^{1-d}) \text{ as } r \to 0.
$$
\end{lemma}

\begin{proof} 
To prove the first statement, we note that if $u \in  \Dom \LapoD$ then certainly $u$ is in the form domain, which implies that there exists a sequence $\phi_j \in C_c^\infty(0, \infty)$ with $\phi_j \to u$ under the form domain norm. Using Sobolev this implies that $\phi_j(r) \to u(r)$ pointwise for all $r > 0$. But by Cauchy-Schwartz,
$$
\Big( \int_0^r \phi'_j(s) \, ds \Big) \leq \Big( \int_0^r s^{1-d} \, ds \Big) \Big( \int_0^r (\phi'_j(s))^2 s^{d-1} \, ds \Big).
$$
Since the last bracket is bounded by the form domain norm squared, we have 
$$
\Big( \frac{\phi_j(r)}{r^{(2-d)/2}} \Big)^2 \leq \| \phi_j \|_Q^2,
$$
and the right hand side is uniformly bounded. Hence
$$
u(r) = O(r^{(2-d)/2}).
$$

To prove the second statement, notice that if $u \in \Dom \LapoN$ then for all $g \in C_c^\infty[0, \infty)$ we have 
$$
\int_0^\infty  f'(s) g'(s)  s^{d-1} \, ds = \int_0^\infty
(-f''(s) - (d-1)/s f'(s)  ) g(s) \, ds.
$$
In particular, we have 
$$
\lim_{\epsilon \to 0} \int_\epsilon^\infty \Big(  f'(s) g'(s) + (f''(s) + (d-1)/s f'(s)) g(s) \Big) s^{d-1} \, ds = 0.
$$
This shows that 
$$
\lim_{\epsilon \to 0}  f'(\epsilon) g(\epsilon) \epsilon^{d-1} = 0,
$$
and choosing $g$ with $g(0) \neq 0$, we obtain $f'(\epsilon) = o(\epsilon^{1-d})$.
\end{proof}

Using this Lemma, and the asymptotics of the functions $k, l, k', l'$ from the previous section, we see that for $1 < d < 2$, the kernel of $(\LapoN + \lambda^2)^{-1}$, $\lambda > 0$, is given by 
\begin{equation*} 
K_{(\LapoN + \lambda^2)^{-1}}(x,y) = \left\{ \begin{array}{ll}
   \gamma l(\lambda x)   & \mbox{ if $y\ge x$}\\
   \delta k(\lambda x)  & \mbox{ if $x >y $}.
           \end{array}
    \right.
\end{equation*} 
for some $\gamma, \delta$ (depending on $y$). Indeed, the kernel must be a linear combination of $k(\lambda x)$ and $l(\lambda x)$ for $x \neq y$. The absence of $k(\lambda x)$ for $x \geq y$ follows from Lemma~\ref{dom}, while the absence of $l(\lambda x)$ for $x \geq y$ follows from the exponential increase of $l$ at infinity, which is inconsistent with the $L^2$-boundedness of $(\LapoN + \lambda^2)^{-1}$. 
At $x=y$ we impose the conditions of continuity, and that $(\LapoN + \lambda^2) K_{(\LapoN + \lambda^2)^{-1}}(x,y) = y^{1-d}\delta_{x-y}$.  This gives the equations for $\LapoN$
\begin{equation}\label{eqoN} 
\left\{ \begin{array}{ll}
\gamma l(\lambda y)=\delta k(\lambda y)\\
\gamma l'(\lambda y)=\delta k'(\lambda y)+\frac{y^{1-d}}{\lambda}
 \end{array}
    \right. .
\end{equation} 
Using \eqref{wron} we see that this has the unique solution 
\begin{equation*} 
\left\{ \begin{array}{ll}
\gamma= \nu\lambda^{d-2}  k(\lambda y)\\
\delta= \nu\lambda^{d-2}  l(\lambda y)
 \end{array}
    \right. 
\end{equation*} 
where $\nu$ is as in \eqref{wron}, so the resolvent kernel for $\LapoN$ is 
\begin{equation}\label{laponker} 
K_{(\LapoN + \lambda^2)^{-1}}(x,y) = \left\{ \begin{array}{ll}
   \nu\lambda^{d-2}  k(\lambda y) l(\lambda x)   & \mbox{ if $y\ge x$}\\
   \nu\lambda^{d-2}  l(\lambda y)k(\lambda x)  & \mbox{ if $x >y $}.
           \end{array}
    \right.
\end{equation} 
It is not hard to see that this formula is valid for all $d$. 

On the other hand,  Lemma~\ref{dom} implies that the kernel of $(\LapoD + \lambda^2)^{-1}$ is given by 
\begin{equation}\label{lapodker} 
K_{(\LapoD + \lambda^2)^{-1} }(x,y) = \left\{ \begin{array}{ll}
   \gamma \big(-A(0) k(\lambda x)+ l(\lambda x) \big)   & \mbox{ if $y\ge x$}\\
   \delta k(\lambda x)  & \mbox{ if $x >y $}.
           \end{array}
    \right.
\end{equation} 
for some $\gamma, \delta$ (depending on $y$). 
The same calculation gives 
\begin{equation*} 
\left\{ \begin{array}{ll}
\gamma= \nu\lambda^{d-2}  k(\lambda y)\\
\delta= \nu\lambda^{d-2} \Big( l(\lambda y) - A(0) k(\lambda y) \Big)
 \end{array}
    \right.
\end{equation*} 
so the resolvent kernel for $\LapoD$, $1 < d < 2$ is 
\begin{equation*} 
K_{(\LapoD + \lambda^2)^{-1}}(x,y) = \left\{ \begin{array}{ll}
   \nu\lambda^{d-2} \Big( k(\lambda y) l(\lambda x) - A(0) k (\lambda x) k(\lambda y) \Big)  & \mbox{ if $y\ge x$}\\
   \nu\lambda^{d-2} \Big( l(\lambda y)k(\lambda x)- A(0) k (\lambda x) k(\lambda y) \Big)  & \mbox{ if $x >y $}.
           \end{array}
    \right.
\end{equation*} 
Thus it differs from the kernel for $\LapoN$ by a rank one term. We shall show that this rank one term is responsible for different boundedness properties of the Riesz transform --- compare Theorems~\ref{rton} and \ref{rtod}. 

\begin{remark} Here and in the resolvent computations below, we use
\eqref{wron} to simplify the expressions.
\end{remark}

\begin{remark} The operators $\LapoN$ $\LapoD$ and $\LaprD$  and the corresponding resolvent kernels are homogeneous: that is, 
$$
K_{ (\LapoD+\lambda^2)^{-1}}(x,y)=\lambda^{d-2}K_{(\LapoD +1)^{-1}}(\lambda x,\lambda y),
$$ 
and the same relation holds for $\LapoN$ and $\LaprD$.
\end{remark}

\subsection{ The resolvent for $ \LapiD $} 
For the Dirichlet boundary condition on $[1, \infty)$, the kernel
$K_{(\LapiD+\lambda^2)^{-1}}$ has the following structure:
\begin{equation} 
K_{(\LapiD+\lambda^2)^{-1}}(x,y) = \left\{ \begin{array}{ll}
   \beta k(\lambda x)+\gamma l(\lambda x)   & \mbox{ if $x \leq y$}\\
   \delta k(\lambda x)  & \mbox{ if $x >y \ge 1$}.
           \end{array}
    \right.
\label{k-lapid}\end{equation} 
Arguing as above, we obtain the following equations:
\begin{equation*} 
\left\{ \begin{array}{ll}
0 =\beta +\gamma A(\lambda)\\
\beta +\gamma A(\lambda y)=\delta \\
\beta +\gamma B(\lambda y)=\delta+\frac{y^{d-1}}{\lambda k'(\lambda y)}.
 \end{array}
    \right.
\end{equation*}  
Hence
\begin{equation*} 
\left\{ \begin{array}{ll}
\gamma=-\frac{ y^{d-1} }{ \lambda k'(\lambda y)  D(\lambda y)}\\
\beta = \frac{y^{d-1}A(\lambda)}{\lambda k'(\lambda y)D(\lambda y)}  \\
\delta=\frac{y^{d-1}A(\lambda)}{ \lambda k'(\lambda y)D(\lambda y)} -
\frac{y^{d-1}A(\lambda y)}{ \lambda k'(\lambda y) D(\lambda y)}.
 \end{array}
    \right.
\end{equation*} 
Thus the kernel of $(\LapiD+\lambda^2)^{-1}$ is given by the formula 
\begin{equation} 
\left\{ \begin{array}{ll}
 \nu{\lambda^{d-2}k(\lambda y)}\left[
 -k(\lambda x)A(\lambda)+
l(\lambda x) \right]    &    \mbox{if} \quad1\le x\le y \\
 \nu{\lambda^{d-2}k(\lambda x)}\left[
-k(\lambda y)A(\lambda)+l(\lambda y)  \right]  &    \mbox{if} \quad 1 \le  y \le x.
 \end{array}
    \right.
\label{res-lapid}\end{equation} 

\subsection{ The resolvent for $\LapiN$}
Again the kernel must have the structure \eqref{k-lapid}. From the Neumann boundary condition we obtain equations
\begin{equation*} 
\left\{ \begin{array}{ll}
0 =\beta +\gamma B(\lambda)\\
\beta +\gamma A(\lambda y)=\delta \\
\beta +\gamma B(\lambda y)=\delta+\frac{y^{d-1}}{\lambda k'(\lambda y)}.
 \end{array}
    \right.
\end{equation*} 
The solution is
\begin{equation*} 
\left\{ \begin{array}{ll}
\gamma=-\frac{ y^{d-1} }{ \lambda k'(\lambda y)  D(\lambda y)}\\
\beta = \frac{y^{d-1}B(\lambda)}{\lambda k'(\lambda y)D(\lambda y)}  \\
\delta=\frac{y^{d-1}B(\lambda)}{ \lambda k'(\lambda y)D(\lambda y)}-
\frac{y^{d-1}A(\lambda y)}{ \lambda k'(\lambda y) D(\lambda y)}.
 \end{array}
    \right. 
\end{equation*} 
The kernel of $(\LapiN+\lambda^2)^{-1}$ is thus given by
\begin{equation*} 
%K_{(\LapiN+\lambda^2I)^{-1}}(x,y)
\left\{ \begin{array}{ll}
 \nu{\lambda^{d-2}k(\lambda y)}\left[
 -k(\lambda x)B(\lambda)+
l(\lambda x) \right]    &    \mbox{if} \quad1\le x\le y \\
 \nu{\lambda^{d-2}k(\lambda x)}\left[
-k(\lambda y)B(\lambda)+l(\lambda y)  \right]  &    \mbox{if} \quad 1 \leq  y \le x.
 \end{array}
    \right.
\end{equation*}

\subsection{ The resolvent for $\Lapd $}

Now we calculate the resolvent for the operator $\Lapd$.
The kernel $K_{(\Lapd+\lambda^2)^{-1}}\colon \widetilde{\R}\times\widetilde{\R}$ necessarily has the following structure. 
For $y \ge 1$, 
\begin{equation*} 
K_{(\Lapd+\lambda^2)^{-1}}(x,y) = \left\{ \begin{array}{ll}
   \alpha k(\lambda|x|) & \mbox{ if $x \le -1$}\\
   \beta k(\lambda x)+\gamma l(\lambda x)   & \mbox{ if $y\ge x$}\\
   \delta k(\lambda x)  & \mbox{ if $x >y \ge 1$}.
           \end{array}
    \right.
\end{equation*} 
The kernel is defined for $y \leq -1$ by the condition $K_{\lambda}(x,y)=K_{\lambda}(-x,-y)$, which follows by uniqueness of the resolvent kernel. 
The equation
\begin{equation}\label{abcd}
(\lambda^2+\Lapd) K_{\lambda}(x,y)=    y^{1-d}\delta_{(x-y)}
\end{equation}
gives the following equations for the coefficients $\alpha, \beta, \gamma, \delta$:
\begin{equation}\label{eq} 
\left\{ \begin{array}{ll}
\alpha k(\lambda)=\beta k(\lambda)+\gamma l(\lambda)\\
-\alpha  k'(\lambda)=\beta k'(\lambda)+\gamma l'(\lambda)\\
\beta k(\lambda y)+\gamma l(\lambda y)=\delta k(\lambda y)\\
\beta k'(\lambda y)+\gamma l'(\lambda y)=\delta k'(\lambda y)+\frac{y^{1-d}}{\lambda}.
 \end{array}
    \right.
\end{equation} 
This has solution
\begin{equation*} 
\left\{ \begin{array}{ll}
\gamma= \nu\lambda^{d-2}  k(\lambda y)\\
\alpha  = -\nu{\lambda^{d-2}k(\lambda y)D(\lambda)}/2 \\
\beta = -\nu{\lambda^{d-2}k(\lambda y)C(\lambda)}/{2 }  \\
\delta=-\nu{\lambda^{d-2}k(\lambda y)C(\lambda)}/{2}+
\nu{\lambda^{d-2}k(\lambda y)A(\lambda y)}.
 \end{array}
    \right.
\end{equation*} 
Thus

\begin{lemma}\label{res-lapd}
The kernel of the resolvent  operator $K_{(\Lapd+\lambda^2)^{-1}}(x,y)$ is given by the formula
\begin{equation*} 
K_{(\Lapd+\lambda^2)^{-1}}(x,y)=\left\{ \begin{array}{ll}
 -\nu{\lambda^{d-2}k(\lambda y)k(\lambda |x|)D(\lambda)/2}   & \mbox{if} \quad x\le -1\\
 \nu{\lambda^{d-2}k(\lambda y)}\left[
 -k(\lambda x)C(\lambda)/2+
l(\lambda x) \right]    &    \mbox{if} \quad1\le x\le y \\
 \nu{\lambda^{d-2}k(\lambda x)}\left[
-k(\lambda y)C(\lambda)/2+l(\lambda y)  \right]  &    \mbox{if} \quad y \le x.
 \end{array}
    \right.
\end{equation*} 
for all $y\ge 1$. For $y\le -1$ we calculate the kernel using the identity
$K_{(\Lapd+\lambda^2)^{-1}}(x,y)\linebreak=K_{(\Lapd+\lambda^2)^{-1}}(-x,-y)$. 
\end{lemma}

\subsection{Resolvent for $\LaprD$} This may be obtained formally by considering the set $\RR_b = (-\infty, -b] \cup [b, \infty)$, $b > 0$, and sending $b$ to $0$. Then $C(\lambda)$ and $D(\lambda)$ are replaced by $C(b\lambda)$ and $D(b\lambda)$ in the formulae above.
Noting that $C(0) = D(0)$, the resolvent kernel for $(\LaprD+\lambda^2)^{-1}$ then is given by 
\begin{equation*} 
\left\{ \begin{array}{ll}
 -\nu{\lambda^{d-2}k(\lambda y)k(\lambda |x|)D(0)/2}   & \mbox{if} \quad x\le 0\\
 \nu{\lambda^{d-2}k(\lambda y)}\left[
 -k(\lambda x)D(0)/2+
l(\lambda x) \right]    &    \mbox{if} \quad 0\le x\le y \\
 \nu{\lambda^{d-2}k(\lambda x)}\left[
-k(\lambda y)D(0)/2+l(\lambda y)  \right]  &    \mbox{if} \quad y \le x.
 \end{array}
    \right.
\end{equation*} 
for all $y\ge 0$. For $y\le 0$ we calculate the kernel using the identity
$K_{(\LaprD+\lambda^2)^{-1}}(x,y)=K_{(\LaprD+\lambda^2)^{-1}}(-x,-y)$.

%%%%%%%%%%%%%%%%%%%%%%%%%%%%%%%%%%%

\section{Riesz transforms}\label{RT}

\subsection{Riesz transform for $\LapoN$.}

Using the formula
\begin{equation}\label{rtfor}
L^{-1/2} = \frac1{\pi} \int_0^\infty (L + \lambda^2)^{-1} \, d\lambda
\end{equation}
(valid for positive operator $L$) we analyze the boundedness of the Riesz transform of the operators defined  in Section \ref{opdef} on $L^p$. 
We start our discussion of the Riesz transform with the  operators $\LapoN$
acting on $L^2(r^{d-1}dr,(0,\infty))$. It turns out that for these operators the Riesz transform 
is bounded for all $1<p< \infty$, for all $d > 1$.

\begin{theorem}\label{rton}
The Riesz transform $\nabla [\LapoN]^{-1/2}$ is bounded on all $L^p(\R_+,r^d dr)$ spaces for
all $d>1$ and $1<p<\infty$. In addition the operator $\nabla [\LapoN]^{-1/2}$ is of weak type $(1,1)$. 
\end{theorem}
\begin{proof}
By (\ref{rtfor}) for $x>y$
$$
  K_{\nabla [\LapoN]^{-1/2}}(x,y)  =   x^{-d} \int_0^\infty \lambda^{d-1}k'(\lambda) l(\lambda \frac{y}{x}) d\lambda,  $$
   while for $x < y$
\begin{eqnarray*}
K_{ \nabla [\LapoN]^{-1/2}}(x,y)&=&\int_0^\infty\lambda^{d-1}l'(\lambda x) k(\lambda y) d\lambda
\\&=& y^{-d} \int_0^\infty\lambda^{d-1}l'(\lambda \frac{x}{y}) k(\lambda ) d\lambda .
\end{eqnarray*}
In both cases $K(x,y)$ is given by $y^{-d}$ times a function of $x/y$. Now consider the isometry $M : L^p(\R_+, r^{d-1} \, dr) \to L^p(\R_+, r^{-1} \, dr)$ defined by 
$$
(Mf)(x)=  x^{d/p} f(x).
$$
The corresponding operator has kernel 
$$ \tilde K(x,y) =  x^{d/p} K_{ \nabla [\LapoN]^{-1/2}}(x,y) y^{d - d/p}
$$
and is a function of $x/y$ (depending on parameters $d$ and $p$). Now change variable to $s = \log x$; this induces an isometry from $L^p(\R_+, r^{-1} \, dr)$ to $L^p(\R, ds)$ and in this picture, $\tilde K$ becomes a convolution kernel $u(s-t)$, with $u$ (depending on $d$ and $p$) smooth except at $s=0$. To determine the boundedness of the convolution kernel on $L^p(\R)$ we need to analyze the behaviour of $u(s)$ as $s \to \pm \infty$ and as $s \to 0$. 

First consider asymptotics as $s \to \pm\infty$, which is equivalent to $x/y \to 0$ or $\infty$. Using the asymptotics (A), (B) and (C)   of Section \ref{asym}
we see that if $y > 2x$, then 
\begin{eqnarray*}
|K_{\nabla [\LapoN]^{-1/2}}(x,y)| \le x^{-d}
\int_0^\infty\lambda^{d-1}    \lambda^{1-d}  e^{-\lambda/2} d\lambda
\le C x^{-d},
\end{eqnarray*}
while for $x > 2y$ we have using asymptotics (A), (B) and (C)    
\begin{eqnarray*}
|K_{\nabla [\LapoN]^{-1/2}}(x,y)| \le y^{-d}\frac{x}{y}
\int_0^\infty\lambda^{d-1}    \lambda^{1-d}  e^{-\lambda/2} d\lambda
\le C xy^{-d-1}.
\end{eqnarray*}
In terms of the kernel $u$ this gives exponential decay as $s \to \pm \infty$ for all $d > 1$ and $1 < p < \infty$. For example, if $s \to +\infty$, then $y/x \to 0$ and we have from the first asymptotics
$\tilde K(x,y) \sim (y/x)^{d(1 - 1/p)}$, or $u(s) \sim e^{-sd(1 - 1/p)}$. 

Next we analyze the behaviour of $u(s)$ near $s=0$, corresponding to near $x=y$ in the original coordinates. If  $\frac{1}{2} \le \frac{x}{y} < 1$ then  by Asymptotics (A), (B) and (C)  of  Section \ref{asym}, as well as Remark~\ref{asympt}, we have
\begin{eqnarray*}
\left| \tilde K(x,y)-\frac{b}{x-y}\right|= 
\left| ((x/y)^{d/p} \int_0^\infty\lambda^{d-1}k'(\lambda) l(\lambda \frac{y}{x}) d\lambda  -\frac{b}{x-y}\right|
\\  \le C+\left| (x/y)^{d/p}\int_1^\infty\left[\lambda^{d-1}k'(\lambda) l(\lambda \frac{y}{x})-be^{\lambda(1-\frac{y}{x})}\right] d\lambda \right|
\\ \le C \log|1-\frac{x}{y}|
\end{eqnarray*}
where $b=\lim_{\lambda \to \infty}\lambda^{d-1}l(\lambda)k(\lambda)$. 
A similar calculation shows that the above estimates hold also if $1 < \frac{x}{y} \le 2$. Thus, choosing some function $\phi(s) \in C_c^\infty(\R)$ which is identically $1$ near $s=0$, we can write $u(s) = \phi(s)/s + \tilde u(s)$, where $\tilde u\in L^1(\R)$. The first term is a Calderon-Zygmund kernel which is bounded on $L^p$ for $1 < p < \infty$ and of weak type $(1,1)$, while the second is bounded on $L^p$ for all $1 \leq p \leq \infty$.  
\end{proof}

\begin{remark}\label{remainder} We observe from the computations in Section~\ref{res-kernels} that the kernel of the resolvent of $\LapoN$ is a summand of the expression for the kernel of the resolvent of all our other operators in  Section~\ref{res-kernels}. Hence, in view of Proposition~\ref{rton},  to determine the $L^p$ boundedness of our other operators, we can subtract the kernel of the resolvent of $\LapoN$ and consider  only  the remainder. We call this the ``$kk$'' part of the kernel since it is bounded by a multiple (depending on $\lambda$) of $k(\lambda x) k(\lambda y)$. 
\end{remark}

%%%%%%%%%%%%%%%%
\subsection{Riesz transform for $\LapoD$ and $\LaprD$.}
We recall that we only consider these operators in the range 
$1 < d < 2$.

\begin{theorem}\label{rtod}
 The Riesz transforms $T_{\LapoD} = \nabla [\LapoD]^{-1/2}$ and $T_{\LaprD} = \nabla [\LaprD]^{-1/2}$  corresponding to the operators
$\LapoD$ and $\LaprD$ respectively  are bounded  on the space  $\Lpz$ precisely for $1 < p < d/(d-1)$. 
\end{theorem}
\begin{proof}
By Remark~\ref{remainder}, to prove Theorem~\ref{rtod} it is enough to show that 
the part of Riesz transform corresponding to the ``$kk$'' part of the kernel is bounded on $L^p((0,\infty); r^{d-1}dr)$ if
and only if  $1 < p < d/(d-1)$. Since the argument is similar for both, we only write down the proof for $T_{\LapoD}$. 

The ``$kk$'' part of the kernel is $A(0)$ times 
$$
x^{-d}\int_0^\infty \lambda^{d-1} k'(\lambda ) k(\lambda \frac{y}{x}) \, d\lambda
=y^{-d}\int_0^\infty \lambda^{d-1} k'(\lambda \frac{x}{y} ) k(\lambda) \, d\lambda.
$$
Using  Asymptotics  (C)  in  Section \ref{asym}  this kernel is bounded above and below by a multiple of 
\begin{equation}\label{kur}
\begin{cases}
C x^{-d} \text{ for } y < x , \\
C y^{-1} x^{1-d} \text{ for } x \leq  y.
\end{cases}
\end{equation}
We consider the corresponding convolution kernel $u = u_{d,p}$ as in the proof of Proposition~\ref{rton}.  This is  
$$
\begin{cases}
C   e^{d(t-s)(1/p - 1)}  \text{ for } s < t , \\
C  e^{(s-t)(d/p-d+1)} \text{ for } s > t.
\end{cases}
$$
This is bounded on $L^p((0, \infty); r^{d-1} dr)$
if and only if $d/p-d+1<0$. This proves Theorem~\ref{rtod} for Dirichlet Laplacian $\LapoD$.  
The proof for the operator $\LaprD$ is essentially identical with 
$A(0)$ replaced by $D(0)$. 
\end{proof}

\subsection{Riesz transform for $\LapiN$, $\Lapd$ and $\LapiD$.}

We begin with a lemma on the $L^p$ boundedness of kernels satisfying certain pointwise bounds. 

\begin{lemma}\label{kernel-bound} Consider the kernel $K(x,y)$ defined by 
\begin{equation}
K(x,y) = \begin{cases} x^{-\alpha} y^{-\beta} , \quad x \leq y \\
                                  x^{-\alpha'} y^{-\beta'} , \quad x > y
                                  \end{cases}
\label{K}\end{equation}
If  $\alpha + \beta > d$, $\alpha' + \beta' > d$ and 
$$
\frac{d}{\min(d,\alpha)} < p < \frac{d}{\max(0,d-\beta)}
$$
then $K$ is bounded as an operator on $L^p([1, \infty); r^{d-1} dr)$.
\end{lemma}

\begin{proof} This result is essentially contained in Proposition 5.1 of \cite{GH1}, but for completeness we give the proof here. It is sufficient to prove that the operator with kernel 
$$K(x,y) =  \begin{cases} x^{-\alpha} y^{-\beta}, \quad x \leq y \\
0, \quad x > y \end{cases}
$$
is bounded on $L^p([1, \infty), r^{d-1} dr)$ for $p < d/(\min(0,d-\beta))$, since the other part follows by duality. We compute 
\begin{equation*}\begin{gathered}
\| Kf \|_p^p = \int_1^\infty x^{-p\alpha} \Big| \int_x^\infty y^{-\beta} f(y) y^{d-1} \, dy \Big|^p \, x^{d-1}  dx \\
\leq \int_1^\infty x^{-p\alpha} \Big( \int_x^\infty  |f(y)|^p y^{d-1} \, dy \Big)
\Big( \int_x^\infty y^{-p' \beta} y^{d-1} dy \Big)^{p/p'} \, x^{d-1} dx \\
= C \Big( \int_1^\infty x^{-p\alpha + d-1 - p\beta +(p-1)d} \, dx \Big) \| f \|_p^p
\leq C \| f \|_p^p,
\end{gathered}\end{equation*}
where we used $-p'\beta + d < 0$ (which is equivalent to $p < d/\max(0, d-\beta)$) for convergence of the $y$-integral, and $\alpha + \beta > d$ for convergence of the $x$-integral. 
\end{proof}

\begin{theorem}\label{noj}
The Riesz transform $\TiN = \nabla [\LapiN]^{-1/2}$ is bounded on the spaces $\Lpj$  for all  $d>1$ and for all $1<p <\infty$. 
\end{theorem}
\begin{proof}
Following Remark~\ref{remainder} we only consider the ``$kk$'' part of the kernel. This is 
\begin{equation}
\int_0^\infty \lambda^{d-1}k(\lambda y)k'(\lambda x)F(\lambda) \, d\lambda
\label{lint}\end{equation}
where $F(\lambda) = B(\lambda)$. (We write this proof in such a way that it is easily adapted to treat the other operators in Section~\ref{opdef}; hence later we shall consider other cases $F = A, C, D$.) We break this integral into a `large $\lambda$' piece and a `small $\lambda$' piece. The large $\lambda$ piece is 
\begin{equation}
\int_{1/\min(x,y)}^\infty \lambda^{d-1}k(\lambda y)k'(\lambda x)F(\lambda) \, d\lambda.
\label{ll}\end{equation}
To analyze this we assume asymptotics $F(\lambda) \sim e^{2\lambda}$ for $\lambda \geq 1$ and $F(\lambda) \sim \lambda^\beta$ for $\lambda \leq 1$; similarly we write $k(\lambda) \sim \lambda^{-\gamma}$ for $\lambda \leq 1$, which is valid for every $d \neq 2$. 
Then, using the large $\lambda$ asymptotics for $k$ and $k'$, and bounding $F(\lambda)$ by $e^{2\lambda}$ which is valid for every $\lambda > 0$, we estimate the integral \eqref{ll} by
\begin{equation}\begin{gathered}
\int_{1/\min(x,y)}^\infty \lambda^{d-1} (\lambda y)^{-(d-1)/2} e^{-\lambda y} (\lambda x)^{-(d-1)/2} e^{-\lambda x} e^{2\lambda} \, d\lambda \\
= (xy)^{(1-d)/2} \int_{1/\min(x,y)}^\infty e^{-(x+y-2)\lambda} \, d\lambda = \frac{e^{-(x+y-2)/\min(x,y)}}{x+y-2 } (xy)^{-(d-1)/2}.
\end{gathered}\end{equation}

For $x + y \leq 4$, this is essentially the kernel $1/(s+t)$ on the half-line $\R_+$ which is known to be bounded on all $L^p$ spaces, $1 < p < \infty$ (\cite{HLP}, sec. 9.1). For $x + y \geq 4$, we can estimate $e^{-(x+y-2)/\min(x,y)}$ by $(x+y/\min(x,y))^{-N}$ for arbitrary $N$, and is thus bounded by both $(x/y)^{-N}$ and $(y/x)^{-N}$. This is bounded  on all $L^p$ spaces, $1 < p < \infty$, by Lemma~\ref{kernel-bound}.

To treat the small $\lambda$ part of \eqref{lint} we split into cases $x \leq y$ and $y \leq x$. Let us first consider $x \leq y$. Then we break up the integral $\int_0^{1/x}$ into $\int_0^{1/y} + \int_{1/y}^{1/x}$. 
To estimate the first integral we use the small variable asymptotics for both $k'(\lambda x)$ and $k(\lambda y)$, while in the second we use the small  variable asymptotics for $k'(\lambda x)$ and the large variable for $k(\lambda y)$. The first estimate then is
\begin{eqnarray*}
\int_0^{1/y}  \lambda^{d-1}  (\lambda x)^{1-d} (\lambda y)^{-\gamma} \lambda^{\beta}     d\lambda = 
x^{1-d} y^{-\gamma} \int_0^{1/y} \lambda^{\beta - \gamma} d\lambda =  x^{1-d} y^{-\beta - 1} .
\end{eqnarray*}

The second estimate is 
\begin{equation}\begin{gathered}
\int_{1/y}^{1/x}  \lambda^{d-1} (\lambda x)^{1-d} (\lambda y)^{(1-d)/2} e^{-\lambda y} \lambda^{\beta}     d\lambda = 
 x^{1-d} y^{-\beta - 1} \int_{1}^{y/x} \lambda^{\beta-d/2+3/2} e^{-\lambda} d\lambda
\\  \leq C  x^{1-d}  y^{-\beta-1}.
\end{gathered}\end{equation}
Here we changed variable to $\lambda y$ in the integral, and used the integrability   of the function $\lambda^{\beta-d/2+3/2} e^{-\lambda}$ on $(1, \infty)$. 

For $y \leq x$, we split the integral $\int_0^{1/y}$ into $\int_0^{1/x} + \int_{1/x}^{1/y}$. 
To estimate the first integral we use the small variable asymptotics for both $k'(\lambda x)$ and $k(\lambda y)$, while in the second we use the large variable asymptotics for $k'(\lambda x)$ and the small variable for $k(\lambda y)$. The first estimate then is
\begin{eqnarray*}
\int_0^{1/x}  \lambda^{d-1}  (\lambda x)^{1-d} (\lambda y)^{-\gamma} \lambda^{\beta}     d\lambda = 
x^{1-d} y^{-\gamma} \int_0^{1/x} \lambda^{\beta - \gamma} d\lambda =  x^{\gamma - \beta-d} y^{-\gamma} .
\end{eqnarray*}

The second integral is estimated by
\begin{equation}\begin{gathered}
\int_{1/x}^{1/y}  \lambda^{d-1} (\lambda x)^{(1-d)/2} e^{-\lambda x} (\lambda y)^{-\gamma}  \lambda^{\beta}     d\lambda = 
 x^{\gamma - \beta-d} y^{-\gamma}   \int_{1}^{x/y} \lambda^{\beta-\gamma + (d-1)/2} e^{-\lambda} d\lambda
\\  \leq C x^{\gamma - \beta-d} y^{-\gamma}  .
\end{gathered}\end{equation}

To summarize, the `small $\lambda$' part of the kernel is bounded by 
\begin{equation}
\begin{cases}
x^{1-d} y^{-\beta - 1}, \quad x \leq y \\
x^{\gamma - \beta-d} y^{-\gamma}, \quad x \geq y.
\end{cases}
\label{sl}\end{equation}
Now in the case of interest, we have $\beta = d$ and $\gamma = \max(0, d-2)$ for $d \neq 2$, while the case $d=2$ has extra logarithmic terms. For $d \neq 2$, then, by multiplying by suitable positive powers of $y/x$ when $x \leq y$ or $x/y$ when $x \geq y$, we see that the kernel is bounded by $x^{-d} y^{-d}$, while for $d=2$ it is easy to see that the kernel is bounded by $x^{-2} y^{-2+\epsilon}$ for any $\epsilon > 0$. Boundedness on $L^p$ for all $1 < p < \infty$ then follows from  Lemma~\ref{kernel-bound}. 
\end{proof}

\begin{theorem}\label{main}
Riesz transforms $\TiD = \nabla [\LapiD]^{-1/2}$  and    $\Td = \nabla [\Lapd]^{-1/2}$ are bounded on  $\Lpj$  and on $\Lpwr$ respectively if and only if 
\begin{enumerate}
\item[(i)] $1<p<d$ for $d>2$
\item[(ii)] $1<p \le 2$ for $d=2$
\item[(iii)] $1<p < \frac{d}{d-1}$ for $1 < d < 2$.
\end{enumerate}
\end{theorem}

\begin{proof} Following Remark~\ref{remainder}, we only consider the ``$kk$'' part of the kernel of these operators. This part of the kernel has a similar form for both operators; essentially the difference is that the function $A(\lambda)$ in \eqref{res-lapid} gets replaced by $C(\lambda)$ and $D(\lambda)$ in Lemma~\ref{res-lapd}. Since these functions have the same leading asymptotics (see (A), (B), (C) in Section~\ref{specfns}), it is enough to treat one of the operators, so we consider only $\LapiD$ below. 

Consider the expression \eqref{ll} where now $F(\lambda) = A(\lambda)$. The `large $\lambda$' piece is treated exactly as below
\eqref{ll}, so it suffices to consider the `small $\lambda$' piece. We now break into cases depending on the size of $d$ relative to $2$.

\textbf{Case $d > 2$.} In this case, $\beta =\gamma = d-2$, so by \eqref{sl}, the small $\lambda$ piece is bounded by 
\begin{equation}
\begin{cases}
x^{1-d} y^{1-d}, \quad x \leq y \\
x^{-d} y^{2-d}, \quad x \geq y.
\end{cases}
\label{d>2}\end{equation}
Boundedness on $L^p$ for $1 < p < d$ then follows from  Lemma~\ref{kernel-bound}. 

To show unboundedness on $L^d$, we observe that the upper bound \eqref{d>2} is also a lower bound (using the positivity properties of $k$, $-k'$ and $A$ from Lemma~\ref{pos}). This kernel does not act boundedly on $ (y \log y))^{-1}$ which is in $L^d$,
so boundedness on $L^d$ fails.

\textbf{Case $d < 2$.} In this case, $\beta = \gamma = 0$, so by \eqref{sl}, the small $\lambda$ piece is bounded by 
\begin{equation}
\begin{cases}
x^{-1} y^{-1}, \quad x \leq y \\
x^{-2} , \quad x \geq y
\end{cases}
\label{d<2}\end{equation}
and the result follows from   Lemma~\ref{kernel-bound}. Unboundedness for $p = d/(d-1)$ follows as for the case $d>2$. 

\textbf{Case $d = 2$.} Here we cannot directly use \eqref{sl} since there are logarithmic terms in the expansions of $k(\lambda)$ and $A(\lambda)$ as $\lambda \to 0$. Careful estimation shows that in this case there is a bound on the kernel of the form 
\begin{equation}
\begin{cases}
x^{-1} y^{-1} (\log 2y)^{-1}, \quad x \leq y \\
x^{-2} , \quad x \geq y.
\end{cases}
\label{d=2}\end{equation}
The part of the kernel with $x \geq y$ is bounded on $L^p$ for all $1 < p < \infty$ using the reasoning in the proof of Theorem~\ref{rton}. The part with $x \leq y$ is more delicate, but one can check that the calculation of the proof of part (i) in Lemma~\ref{kernel-bound} can still be made in this case, for $1 < p < 2$, showing boundedness in this range of $p$.  For $p=2$, boundedness is automatic from the equality
$$
\| \nabla f \|_2^2 = \langle \LapiD f, f \rangle = \| (\LapiD)^{1/2} f \|_2^2.
$$
For $p > 2$, we can easily derive the lower bound 
\begin{equation*}
\begin{cases}
x^{-1} y^{-1-\epsilon}, \quad x \leq y \\
x^{-2-\epsilon} , \quad x \geq y
\end{cases}
\end{equation*}
on the kernel, which shows unboundedness for $p > 2$ using the argument above. 
\end{proof}

\begin{comment}
\begin{remark} It is instructive to show the $L^2$ boundedness of $T_{\LapiD}$ when $d=2$ directly from an estimate for the kernel. A more careful treatment shows that for $x \leq y$, the kernel is bounded  by a multiple of 
$x^{-1} (y \log 2y)^{-1}$.

Define $T_k$ to be the operator with kernel $x^{-1} (y \log 2y)^{-1}$ with support $x \leq y, 2^{2^k} \leq x \leq 2^{2^{k+1}}$. Then we can easily compute that
\begin{equation}\begin{gathered}
\| T_k T_l^* \|_{L^2 \to L^2} \leq  2^{-|k-l|/2}, \\
\| T_k^* T_l \|_{L^2 \to L^2} \leq \delta_{kl}
\end{gathered}\end{equation}
and it follows from the Stein-Cotlar Lemma that $\sum_k T_k$ is bounded on $L^2$. The boundedness on $L^2$ is thus quite delicate, which is not surprising considering that $p=2$ is the upper threshold for $L^p$ boundedness. 
\end{remark}
\end{comment}

\begin{remark} Consider the two operators $\LapiN$ and $\LapiD$; we have just shown that the Riesz transform $\TiN$ is bounded on $L^p$ for $1 < p < \infty$, while $\TiD$ is bounded only for $1 < p < d$ if $d > 2$. We can give an explanation for this different range of $p$ which is essentially the same as that given in the introduction of \cite{CCH} 
comparing the Laplacian on $\R^d$ (where the Riesz transform is bounded for all $1 < p < \infty$) and the Laplacian on a manifold with more than one Euclidean end (where it is unbounded for $p \geq d$). 

One notes that the kernel of $L^{-1/2}$, for either $L= \LapiN$ or $\LapiD$, is $\sim f(x) y^{1-d} + O(y^{-d})$ as $y \to \infty$ for fixed $x$. However, the coefficient $f(x)$ of this leading asymptotic is constant in the Neumann case, and nonconstant in the Dirichlet case (in both cases, the leading coefficient is annihilated by $L$, and satisfies the boundary condition at $x=1$, so it cannot be constant in the Dirichlet case). Hence, after applying $\nabla$ on the left, the leading coefficient vanishes in the Neumann case, resulting in the leading behaviour for $\TiN$ being $O(y^{-d})$ while for $\TiD$ it is still $\sim y^{1-d}$, and this extra decay leads to a larger range of $p$ for $\TiN$ as compared to $\TiD$. 
\end{remark} 

\begin{remark}\label{B(0,r)} Our results, together with the heuristic in the remark above, suggest that for any smooth obstacle $\mathcal{O} \subset \R^d$, the Riesz transform for the Neumann Laplacian on $\R^d \setminus \mathcal{O}$ is bounded for all $1 < p < \infty$, while the Riesz transform for the Dirichlet Laplacian on $\R^d \setminus \mathcal{O}$ is bounded for $1 < p < d$ (including $d$ when $d=2$) and unbounded otherwise. 
We will not pursue this here, but expect that this can be shown using the method of \cite{CCH} and standard potential theory on bounded domains. 
Notice that our results definitely show that the Dirichlet Riesz transform on $\R^d \setminus B(0,r)$ is unbounded for $p \ge d$ and $d >2$ and for all $p>2$ if $d=2$. 
\end{remark}

%%%%%%%%%%%%%%%%%%%%%%%%%%%%%%%%%%%%%

\section{Generalizations}
In this section we consider several generalizations of our setup. 
First we consider the effect of adding a potential function to our operator which is either (i) an inverse-square potential, i.e. a constant times $r^{-2}$ or (ii) a delta-function potential. For simplicity we consider only the operators $\LapoN$ and $\Lapd$. In both of these cases, the previous arguments work with minor modifications, i.e. we can write down the exact expression for the kernel of the resolvent and use it to determine an expression for the kernel of the Riesz transform. 

\subsection{Inverse-square potentials}\label{isqp}
Consider the quadratic forms \eqref{Q1} and \eqref{Q2} with the term
$$
\int   f(r)g(r) \frac{c}{r^2} \, r^{d-1} dr
$$
added. The result is a positive quadratic form provided that the constant $c$ is greater than $-(d-2)^2/4$, which we shall always assume. The operator is then 
$$
Lf = -f'' - \frac{d-1}{r} f' + \frac{c}{r^2} f.
$$
Following the reasoning of Section~\ref{specfns},  if $f$ solves 
$$
f''(r) + \frac{d-1}{r} f'(r) - \frac{c}{r^2} f(r) = \lambda^2 f(r)
$$
then $F(r) = r^{d/2 - 1} f(r/\lambda)$ solves 
$$
r^2 F''(r) + r F'(r) - \big( r^2 + (d/2 - 1)^2 + c \big) F(r) = 0.
$$
Define the number $d' = d'(d, c)$ to be the positive root of the equation 
$(d'/2 - 1)^2 = (d/2 - 1)^2 + c$. 
Then $f(r)$ is a linear combination of the functions $r^{-d/2 + 1} I_{d'/2 - 1}(\lambda r)$ and $r^{-d/2 + 1} K_{d'/2 - 1}(\lambda r)$, where $I_\nu$ and $K_\nu$ are modified Bessel functions. We can define solutions $k = k_{d,c}(\lambda)$ and $l = l_{d,c}(\lambda)$ by their asymptotics at $\lambda = 0$, namely (for $c \neq 0$ and for $d > 2$, which is the only case we shall consider here)
\begin{equation*} 
\begin{gathered}
k(\lambda)\approx \left\{ \begin{array}{ll}
\lambda^{2-(d+d')/2}  & \mbox{if} \quad \lambda \le 1\\
\lambda^{(1-d)/2} e^{-\lambda}    &    \mbox{if} \quad \lambda > 1
 \end{array}
    \right.
\\
k'(\lambda)\approx \left\{ \begin{array}{ll}
-\lambda^{1-(d+d')/2}  & \mbox{if} \quad \lambda \le 1\\
-\lambda^{(1-d)/2}e^{-\lambda}    &    \mbox{if} \quad  \lambda >1
 \end{array}
    \right.
\end{gathered} \qquad
\begin{gathered}
l(\lambda)\approx \left\{ \begin{array}{ll}
\lambda^{(d'-d)/2}  & \mbox{if} \quad \lambda \le 1\\
\lambda^{(1-d)/2}e^{\lambda}    &    \mbox{if} \quad 1 \le \lambda
 \end{array}
    \right.
\\
l'(\lambda)\approx \left\{ \begin{array}{ll}
\lambda^{(d'-d)/2-1}  & \mbox{if} \quad \lambda \le 1\\
\lambda^{(1-d)/2}e^{\lambda}    &    \mbox{if} \quad1\le \lambda
 \end{array}
    \right.
\end{gathered}
\end{equation*} 
\begin{equation*} 
A(\lambda) \approx B(\lambda) \approx  D(\lambda)\approx \left\{ \begin{array}{ll}
\lambda^{d'-2}  & \mbox{if} \quad \lambda \le 1\\
e^{2\lambda}    &    \mbox{if} \quad1\le \lambda
 \end{array}
    \right.
\end{equation*} 
%\begin{equation*} 
%B(\lambda)\approx \left\{ \begin{array}{ll}
%\lambda^{d'}  & \mbox{if} \quad \lambda \le 1\\
%e^{2\lambda}    &    \mbox{if} \quad1\le \lambda,
% \end{array}
 %   \right.
%\end{equation*} 
and we have $lk' - k l' = (\nu r^{d-1})^{-1}$ as before. The reasoning in Section~\ref{res-kernels} then goes through verbatim to show that the expression for the resolvent in terms of $k$ and $l$ is exactly as in Section~\ref{res-kernels}, except that now $k$ and $l$ denote $k_{d,c}$ and $l_{d,c}$ respectively. 

For the boundedness of the Riesz transform we consider $\LapoN$ first. 

\begin{proposition}\label{prop-c}
Let $d > 2$ and $c > -(d-2)^2/4$. Then the Riesz transform for $L = \LapoN + c/r^2$ is bounded on $L^p$ precisely in the range 
\begin{equation}
p \in \Bigg( \frac{d}{\min\Big(d,  \frac{d}{2} + 1 + \sqrt{(\frac{d}{2} - 1)^2 + c} \Big)}  , \ 
\frac{d}{\max\Big( 0, \frac{d}{2}  - \sqrt{(\frac{d}{2} - 1)^2 + c} \Big)} \Bigg) 
\label{p-isqp}\end{equation}
where we interpret $d/0 = \infty$. 
\end{proposition}

\begin{proof}
We can use the same reasoning as in Section 6 to transform the kernel into a convolution kernel on $\RR$. The diagonal behaviour is as in the case $c=0$ but the behaviour as $s \to \pm \infty$ is different. We find in this case that the size of the kernel as $x/y \to 0$ and $x/y \to \infty$ is asymptotically given by  
$$
K(x,y) \sim y^{-d} \big( \frac{x}{y} \big)^{-1 + (d'-d)/2} = x y^{-d+1} \big( \frac{x}{y} \big)^{ (d'-d)/2} , \quad \frac{x}{y} \to 0, 
$$
$$
K(x,y) \sim x^{-d}  \big( \frac{y}{x} \big)^{ (d'-d)/2} , \quad \frac{y}{x} \to 0, 
$$
which, as in the proof of Theorem~\ref{rtod}, leads to the conclusion that the Riesz transform is bounded 
(for $c \neq 0$) precisely in the range \eqref{p-isqp}. 
\end{proof}

\begin{remark}
Notice that this range is increasing with $c$, and when $c < 0$,  the lower threshold for $L^p$ boundedness is bigger than $1$ and the upper threshold is less than $d$. Conversely, when $c > 0$, the lower threshold is $1$ and the upper threshold is larger than $d$. Also, when $c \geq d-1$, the upper threshold is $\infty$. 
\end{remark}

\begin{remark} We have previously noted that (for integral $d$) our one-dimensional operator $\LapoN$ is equivalent to looking at radially symmetric functions on $\R^d$. If instead we look at functions which are a radial function times a fixed spherical harmonic with eigenvalue $c$, then this amounts to a one-dimensional problem with inverse-square potential $c/r^2$. A similar comment can be made on any metric cone. 
\end{remark}

\begin{remark}\label{ccc} Following on from the previous remark, our result shows that for the operator $\Delta + c/|x|^2$ on $\R^d$, $c \neq 0$, the Riesz transform cannot be bounded on $L^p$ for $p$ outside the range \eqref{p-isqp}. 
\end{remark}

\begin{remark} Notice that there is a difference in the asymptotics above as compared to Section~\ref{specfns}: when $c \neq 0$,  $l(\lambda) \sim \lambda^{(d-d')/2}$ behaves as a nonzero power as $\lambda \to 0$ and this means that its derivative $(d-d')/2 \lambda^{(d-d')/2 -1}$ vanishes one order \emph{less} as $\lambda \to 0$. However, when $c=0$, then $d=d'$, this leading term vanishes and since $l$ has only even powers in its Taylor series at $\lambda = 0$, this means that the derivative $l'(\lambda)$ vanishes one order \emph{more} at $\lambda = 0$. 
This has the effect of changing the upper threshold for $L^p$ boundedness from $p = d$, the limiting value as $c \to 0$, to $p = \infty$ when $c = 0$. 

This remark, and the previous one, shed light on results of Li \cite{Li} for boundedness of the Riesz transform on cones (as well as the closely-related Theorem 1.3 of \cite{GH1} on asymptotically conic manifolds). If we look just at the functions on the cone which are radial  times a fixed cross-section eigenfunction, then  the constant eigenfunction, with eigenvalue $0$,  corresponds to $c=0$ and we get boundedness for all $1 < p < \infty$. For the eigenfunction corresponding to the smallest nonzero eigenvalue, we get precisely the range of boundedness in Li's theorem, and for all higher eigenvalues, we get a larger range. These considerations make Li's theorem very plausible but do not furnish a proof. 
\end{remark}

%\begin{remark} This result is consistent with the boundedness properties of the Riesz transform deduced for asymptotically conic manifolds with potential function asymptotic to an inverse-square potential in Theorem 1.3 of  \cite{GH1}. \end{remark} 

We next consider the operator $\Lapd$. 

\begin{proposition}
Let $d > 2$ and $c > -(d-2)^2/4$, $c \neq 0$. Then the Riesz transform for $L = \Lapd + c/r^2$ is bounded on $L^p$ precisely for $p$ in the range \eqref{p-isqp}.
\end{proposition}

\begin{proof}
Using Remark~\ref{remainder}, we only have to consider the term 
$$
\int_0^\infty \lambda^{d-1} k'(\lambda x) k(\lambda y) D(\lambda) \, d\lambda.
$$

The `large $\lambda$' piece can be treated exactly as in the proof of Proposition~\ref{rton} and is bounded on $L^p$ for all $1 < p < \infty$. The small $\lambda$ piece can also be treated similarly, setting $\beta = d' - 2$ and $\gamma = (d+d')/2 - 2$, but we have to take into account the different asymptotics of $k'(\lambda)$ as $\lambda \to 0$. This yields a bound on the small $\lambda$ part of 
$$
\begin{cases}
x^{1-(d+d')/2} y^{1-(d+d')/2}, \quad x \leq y \\
x^{-(d+d')/2} y^{2-(d+d')/2}, \quad x \geq y
\end{cases}
$$
and the result follows using  Lemma~\ref{kernel-bound}. 
\end{proof}

\subsection{Delta-function potentials}
We next consider the operator $\Lapd$ for $d > 2$ with a delta-function potential $a \delta_{\pm 1}$ added at the point $\pm 1$, for some $a \in \R$.  Equivalently, we look at the quadratic form \eqref{Q2} with domain 
$$
\{ f \in C^1((\infty, -1] \cup [1, \infty)) \mid f(1) = f(-1), f'(1) - f'(-1) = a \};
$$
let us denote this operator $\Lapda$. 
The interest in this is that for $a = -2(d-2)$ and $d > 4$ we can create an $L^2$ eigenfunction with eigenvalue $0$, namely $|x|^{-(d-2)}$. 

We can compute the resolvent kernel for $\Lapda$ exactly using the approach of Section~\ref{res-kernels}. For $a > -2(d-2)$ it is given by the formula \eqref{res-lapd} with $D(\lambda)$ replaced by 
$$
D_a(\lambda) = \frac{\lambda^{2-d}}{k(\lambda) (-\lambda k'(\lambda) +a k(\lambda)/2)};
$$
we remark that when $a=0$ this is an alternative expression for $D(\lambda)$ as previously defined. This has the same asymptotics as $\lambda \to 0$ and $\lambda \to \infty$ as $D(\lambda)$, and hence we conclude that for any $a \neq -2(d-2)$, the Riesz transform is bounded on the same range as in Proposition~\ref{main}, that is, $1 < p < d$.

Next consider the case $a = -2(d-2)$. In this case, due to the zero eigenvalue, $(\Lapda)^{-1/2}$ does not exist. Instead we consider $({\Lapda})_+^{-1/2}$, where $(\Lapda)_+= \Lapda~-~\Pi_0$ is the operator projected off the zero eigenspace, and analyze the corresponding Riesz transform $\nabla (\Lapda)_+^{-1/2}$. 

\begin{proposition}
Let $d > 4$, and $a = -2(d-2)$. Then the Riesz transform $\nabla (\Lapda)_+^{-1/2}$ is bounded on $L^p$ for $p$ in the range
\begin{equation}
\frac{d}{d-2} < p < \frac{d}{3}.
\label{ddd}\end{equation}
\end{proposition}

\begin{proof} 
In this case, the asymptotic for $D_a(\lambda)$ as $\lambda \to 0$ is replaced by $\lambda^{d-4} + O(\lambda^{d-2})$. We use the approach of Proposition~\ref{noj}. 

The ``$kk$'' part of the kernel, with the projection included, is 
\begin{equation}
\int_{0}^\infty \Big( \lambda^{d-1}k(\lambda y)k'(\lambda x)D_{-2(d-2)}(\lambda) - c \lambda^{-2} x^{-d+1} y^{-d+2} \Big) \, d\lambda
\label{ll-a}\end{equation}
for some value of $c$. We split the integral as before into the large $\lambda$ piece ($\lambda > 1/\min(x,y)$) and small $\lambda$ piece. The first term in the large $\lambda$ piece is bounded on $L^p$ for all $1 < p < \infty$ exactly as before. The second term in the large $\lambda$ part is bounded by 
\begin{equation}
\begin{cases}
x^{-d+2} y^{-d+2}, \quad x \leq y \\
x^{-d+1} y^{-d+3}, \quad x \geq y.
\end{cases}
\label{aaa}\end{equation}

The small $\lambda$ part can be treated very much as in the proof of Proposition~\ref{rtod} and leads to the same estimate \eqref{aaa}. For brevity we give the argument only for the $\int_0^{1/y}$ part of the kernel when $x \leq y$. In this case we want to estimate the integral 
\begin{equation*}\begin{gathered}
\int_0^{1/y} \Bigg( \lambda^{d-1} \Big( (\lambda x)^{1-d} + O((\lambda x)^{3-d}) \Big)   \Big( (\lambda y)^{2-d} + O(\lambda y)^{4-d} \Big) \times \\ \times \Big( \lambda^{d-4} + O(\lambda^{d-2}) \Big)  - c x^{1-d} y^{2-d} \lambda^{-2}  \Bigg) \, d\lambda.
\end{gathered}\end{equation*}
The divergent, $\lambda^{-2}$ terms necessarily cancel, and the remainder is bounded by $x^{2-d} y^{2-d}$. Boundedness in the range \eqref{ddd} follows immediately from \eqref{aaa} and  Lemma~\ref{kernel-bound}. 
\end{proof} 

\begin{remark} This is consistent with the results in \cite{GH2} when there is an $L^2$ eigenfunction with eigenvalue $0$.
\end{remark}

\subsection{Hodge projector}

In this section we discuss boundedness on $L^p$ spaces of Hodge projectors operators corresponding to operators discussed above. The obtained properties of Hodge projectors are good illustration of our main results. 
It turns out that in one dimensional setting it is easier to study 
Hodge projectors then Riesz transform so we could consider larger family of examples. However,
in this section it is convenient for use equivalent but different notation, which we describe below.

Recall that in Section~\ref{seclap} we define quadratic form $\widetilde{Q_d}$ by the formula
$$
\widetilde{Q_d}(\tilde f \tilde g) = \int_{\tilde \R}\tilde f'(r)\tilde g'(r)|r|^{d-1} dr
$$
and that $\Lapd$ we denote the operator corresponding this form.
Now set $1/d=1-\delta$ and for function $f \colon \widetilde{\R} 
 \to  \C$ we put
 \begin{equation*} 
\tilde{f}(x)= \left\{ \begin{array}{ll}
f(x^d)  & \mbox{if} \quad x  \ge 1\\
f(-|x|^d) &    \mbox{if} \quad x \le -1 .
 \end{array}
    \right.
\end{equation*} 
 Note that
$d\|f\|_{L^p(\widetilde{\R},dx)}^p=\|\tilde f\|_{L^p(\widetilde{\R},r^{d-1}dr)}^p$ and 
that 
$$
\widetilde{Q_d}(\tilde f \tilde g) = \int_{\tilde \R}\tilde f'(r)\tilde g'(r)|r|^{d-1} dr=\frac{1}{d}\int_{\tilde \R} |x|^{2\delta}   f'(x)g'(x) dx=Q'_\delta(f,g),
$$
where the quadratic form $Q'_\delta$ is defined on $L^p(\widetilde{\R},dx)$.
Equivalently we can consider form 
$$
{Q_d}(\tilde f \tilde g) = \int^{\infty}_{-\infty}(1+ |x|)^{2\delta}   f'(x)g'(x) dx
$$
acting on $L^p(\R, dx)$. 
The  above equalities shows that for $1-\delta=1/d$  the boundedness of Riesz transforms  and Hodge projectors corresponding to the operators $\Lapd$ and 
the operator corresponding to the form $Q_d$ are equivalent.

Next assume that  $a\colon \R \to (0,\infty)$ is a Lipschitz continuous function and set
$$
A= \int_\infty^{-\infty} a(s)^{-1} ds.
$$ 
We consider the operator $L=d_x a(x) d_x$ acting on $L^2(\R)$. More precisely we define 
$L_a$ as the Friedrichs extension corresponding to the quadratic form
$$
\langle L_af,g\rangle =Q_a(f,g)=\int_{\R} f'(x)g'(x) a(x) dx
$$
initially defined for all $f, g \in C^{1}_c(\R)$. 
Note that the canonical gradient corresponding to $L$  is given by  
$$
|\nabla f|^2= \frac{1}{2}(L f^2 -2f L_af)=a|f'|^2:
$$
that is,
$$
\nabla f (x) =\sqrt{a(x)} d_x.
$$
%(For any regular strongly local Drichlet quadratic form one can define the gradient $\nabla $ in
%the canonical way using the notion of carr\'e du champ, see \cite{FOT}.)
The adjoint operator $\nabla^*$ is then given by 
$$
\nabla^*=d_x\sqrt{a(x)} .
$$
The Hodge projector corresponding to the operator $L$ is given by the formula  
$\nabla L^{-1}\nabla^*$.
It is well known that  $\nabla L^{-1}\nabla^*$ is a self-adjoint operator and that
$$
(\nabla L^{-1}\nabla^*)^2=\nabla L^{-1}\nabla^*\nabla L^{-1}\nabla^*=\nabla L^{-1}\nabla^*.
$$
That is, the Hodge projector is a projection on $L^2(\R)$. 

\begin{theorem}\label{choc}
If $\frac{1}{a}\notin L^1(\R)$, that is, if $A=\infty$, then the corresponding Hodge projector is equal to the identity operator: $\nabla L^{-1}\nabla^*=\Id$. Otherwise, if  $\frac{1}{a} \in L^1(\R)$, then
$\nabla L^{-1}\nabla^*=\Id-R$ where $R$ is a projection on the function $\frac{1}{\sqrt{a}}$ and so its  kernel is 
given by the formula 
$$
K_R(x,y)=A^{-1} a(x)^{-1/2}a(y)^{-1/2}.
$$
As a consequence  if   $\frac{1}{a} \in L^1(\R)$ and $p \ge 2$  then  the Hodge projector is bounded on   $L^p(\R)$
if and only if the function $x \to a(x)^{-1/2}$ belongs to both  $L^p(\R)$ and $L^{p'}(\R)$, where $1/p+1/p'=1$. Obviously 
if $\frac{1}{a}\notin L^1(\R)$ then the Hodge projector is bounded on all $L^p$ spaces  $1 \le p \le \infty$. 
\end{theorem}

\begin{proof}
The operator $\nabla L^{-1}\nabla^*$ is a projector on $L^2$ so it is uniquely determined by its  kernel. 
Now suppose that there exists a function $f\in L^2(\R)$ such that
$$
\nabla L^{-1}\nabla^*f=0.
$$
Then $L^{-1}\nabla^* f =c$ for some constant $c\in \C$ and 
$\nabla^*f=0$. However if $\nabla^*f=0$ then $(\sqrt{a}f)'=0$ so 
$$
f(x)=\frac{c}{\sqrt{a(x)}}. 
$$
This means that if $\frac{c}{\sqrt{a}}\notin L^2(\R)$ then $\nabla L^{-1}\nabla^*=\Id$. 
Otherwise if $\frac{c}{\sqrt{a}} \in L^2(\R)$ then the kernel of $\nabla L^{-1}\nabla^*$
is a one dimensional space spanned by the function $\frac{1}{\sqrt{a}}$. In this case it is 
clear that the operator $\nabla L^{-1}\nabla^*= \Id + R$
is continuous on $L^p(\R)$ if and only if the function $x \to a(x)^{-1/2}$ belongs to $L^p(\R) \cap L^{p'}(\R)$. 
\end{proof}

\begin{corollary}\label{co}
For $d>2$ the Hodge projector  $d_r \Lapd^{-1} d_r^*$ corresponding to the operator  $\Lapd$ 
is bounded on all  $L^p(|r|^{d-1}dr)$ for all if and only if
$\frac{d}{d-1}<p<d $. In addition $d_r \Lapd^{-1} d_r^*=\Id$ for all $0<\delta \le 1/2$.
\end{corollary}
\begin{proof}
Consider the  family of functions  
 $a_\delta \colon \R \to (0,\infty)$ given by the formula
\begin{equation}\label{delta}
a_\delta(x)=(1+|x|)^{2\delta}
\end{equation}
where $0<\delta <1$.
Note that $a_\delta^{-1/2} \in L^{p'}$ if and only if $p>d$. 
As explained above  for $1-\delta=1/d$  the boundedness of Riesz transforms  and Hodge projectors corresponding to the operators $\Lapd$ and 
$L_a$ with $a$ given by (\ref{delta})  are equivalent. 
Hence Corollary~\ref{co} follows from  Theorem~\ref{choc}.
\end{proof}

We will finish this section with the following consequence of Corollary~\ref{co}. Of course it is already proved in Theorem~\ref{main}. 
\begin{proposition}\label{pro}
For $d>2$ Riesz transform $d_r \Lapd^{-1/2}$ is unbounded on all $L^p(\tilde \R, |r|^{d-1} dr)$ spaces for 
$p\ge d$. 
\end{proposition}
\begin{proof}
The operator $d_r \Lapd^{-1/2}$ is bounded on all $L^p$ for $1<p\le 2$. Hence the adjoint operator
$\Lapd^{-1/2}d_r^*$ is bounded on all $L^p$ spaces for $2\ge p <\infty$ and if Riesz transform 
$d_r \Lapd^{-1/2}$ is bounded on $L^p$ for some $p\ge 2$ then Hodge projector 
$d_r \Lapd^{-1/2 }  \Lapd^{-1/2} d_r^*$ is also bounded on the same space. Thus Proposition~\ref{pro}
follows from  Corollary~\ref{co}.
\end{proof}

%\section{Spectral decomposition}


\begin{thebibliography}{99}


\bibitem[AS]{AS}
{M.~Abramowitz and I.A. Stegun,}
\newblock {\em  Handbook of mathematical functions with formulas, graphs, and
              mathematical tables,}  
\newblock {National Bureau of Standards Applied Mathematics Series}, {55},
U.S. Government Printing Office, Washington, D.C. (1964).

\bibitem[ACDH]{ACDH} {P.~Auscher, T.~Coulhon, X.T.~Duong and S.~ Hofmann,
\newblock Riesz transform on manifolds and heat kernel regularity, \newblock 
\emph{Ann. Sc. E. N. S.}, 37:911--957, (2004). }


\bibitem[Ba1]{Ba1}
{D.~Bakry},
\newblock {Transformations de {R}iesz pour les semi-groupes
              sym\'etriques. {I}. \'{E}tude de la dimension {$1$}},
\newblock in {\em Séminaire de Probabilités XIX} {Lecture Notes in Math.},{1123}:{130--144}, {Springer}, {Berlin}, {(1985)}.

\bibitem[Ba2]{Ba2}
{D.~Bakry},
\newblock {\'{E}tude des transformations de {R}iesz dans les vari\'et\'es
              riemanniennes \`a courbure de {R}icci minor\'ee},
\newblock in {\em Séminaire de Probabilités XIX} {Lecture Notes in Math.},{1247}:{137--172}, {Springer}, {Berlin}, {(1987)}.

\bibitem[BH]{BH}
N. Bouleau and  F. Hirsch,
\newblock  {\em Dirichlet forms and analysis on
  Wiener space}, vol.\ 14 of de Gruyter Studies in Mathematics.
\newblock Walter de Gruyter \& Co., Berlin, (1991).


\bibitem[CZ]{CZ}
A.P. Calder—n and  A. Zygmund, 
\newblock On the existence of certain singular integrals,
\newblock {\em Acta Math.} 88: 85--139, (1952).

\bibitem[CCH]{CCH}
G.~Carron, Th.~Coulhon and  A.~Hassell, 
\newblock Riesz transform and {$L\sp p$}-cohomology for manifolds with
              {E}uclidean ends
\newblock {\em Duke Math. J.  }, 133:59--93, (2006).


\bibitem[CD]{CD}
T. Coulhon and X.~T. Duong,
\newblock Riesz transforms for $1\leq p\leq 2$,
\newblock {\em Trans. Amer. Math. Soc.}, 351(3):1151--1169, (1999).

\bibitem[Da]{Da} \textrm{E. B.~ Davies, 
\newblock {\em Heat kernels and spectral
theory}, \newblock Cambridge University Press, Cambridge, (1989). }


\bibitem[ERSZ]{ERSZ}
A. F.~M. ter~Elst, D.~W~Robinson,~A. Sikora  and Y, Zhu,
  \newblock Second-order operators with degenerate coefficients,
 \newblock {\em Proc. London Math. Soc.} Advance Access published on April 24, (2007).
 
 
 


\bibitem[GH1]{GH1} C.~Guillarmou, A.~Hassell,
\newblock Resolvent at low energy and Riesz transform for Schrodinger operators on asymptotically conic manifolds, I
\newblock {\em preprint,} arXiv:math.AP/0701515,  (2007).

\bibitem[GH2]{GH2} C.~Guillarmou, A.~Hassell,
\newblock Resolvent at low energy and Riesz transform for Schrodinger operators on asymptotically conic manifolds, II
\newblock {\em preprint,} arXiv:math.AP/0703316, (2007).

 \bibitem[HLP]{HLP}
G.~H. Hardy, J.~E. Littlewood, and G.~P{\'o}lya,
\newblock {\em Inequalities},
\newblock Cambridge, at the University Press, (1952).
\newblock 2d ed.

\bibitem[Ka]{Ka}
T. Kato,
\newblock {\em Perturbation theory for linear operators},
\newblock Second edition, Grundlehren der mathematischen Wissenschaften 132.
  Springer-Verlag, Berlin etc., 1984.


\bibitem[Li]{Li}
Hong-Quan Li,
\newblock  La transformation de Riesz sur les vari\'et\'es coniques,
\newblock {\em J. Funct. Anal.}, 168(1):145--238, 1999.

\bibitem[Ri]{Ri} 
M. Riesz, 
\newblock Sur les fonctions conjuguŽes, 
\newblock {\em Math. Zeitschrift,} 27: 218Ð244 (1927).

\bibitem[RS]{RS}
\textrm{ D.~W~Robinson {and} A.~Sikora},
\newblock  Analysis of degenerate elliptic operators  of  Gru\v{s}in type,
\newblock \emph{to appear in Math. Zeitschrift.}.

\bibitem[Ro]{Ro}
\textrm{B.~Roynette,}
  \newblock  In\'egalit\'es li\'ees  du quelques transformations de Riesz [Inequalities tied to some Riesz transformations],
{\em unpublished manuscript},
Université de Nancy 1. Institut Elie Cartan; PB 1993/1, TechReport, CWI Library:137901,
(1993).

\bibitem[Sh1]{Sh1}
Z. Shen.
\newblock ${L}\sp p$ estimates for {S}chr\"odinger operators with certain
  potentials,
\newblock {\em Ann. Inst. Fourier (Grenoble)}, 45(2):513--546, (1995).

\bibitem[Sh2]{Sh2}
Z. Shen,
\newblock Estimates in ${L}\sp p$ for magnetic {S}chr\"odinger operators,
\newblock {\em Indiana Univ. Math. J.}, 45(3):817--841, (1996).


\bibitem[Si]{Si} \textrm{A.~Sikora, \newblock Riesz transform, Gaussian bounds
and the method of wave equation, \emph{Math. Z.}, 247(3):643-662, (2004). }

\bibitem[Ste]{St} \textrm{E. M.~Stein,
\newblock {\em Topics in harmonic analysis related to the Littlewood-Paley theory,} \newblock Princeton University Press, Princeton, N.J., (1970). }

\bibitem[Str]{Str}
R.~S. Strichartz,
\newblock Analysis of the {L}aplacian on the complete {R}iemannian manifold,
\newblock {\em J. Funct. Anal.}, 52(1):48--79, (1983).

\bibitem[Tr]{Tr} \textrm{C. J.~Tranter,
\newblock {\em Bessel functions with some physical applications,} \newblock Hart Publishing Co. Inc., New York, (1969). }



\end{thebibliography}
\end{document}